\documentclass[a4paper,12pt]{amsart}
\pdfoutput=1

\usepackage[sc]{mathpazo}
\usepackage[T1]{fontenc}

\usepackage{graphicx}
\usepackage{amsmath}
\usepackage{amssymb}
\usepackage[active]{srcltx}
\usepackage{hyperref}

\newtheorem{thm}{Theorem}%
\newtheorem{lem}{Lemma}%
\theoremstyle{definition}

\theoremstyle{remark}
\theoremstyle{plain}

\def\CC{{\mathbb C}}

\def\EE{{\mathbb E}}

\def\HH{{\mathbb H}}

\def\NN{{\mathbb N}}
\def\QQ{{\mathbb Q}}

\def\RR{{\mathbb R}}
\def\TT{{\mathbb T}}

\def\ZZ{{\mathbb Z}}

\def\vecm{{\text{\boldmath$m$}}}

\def\vecxi{{\text{\boldmath$\xi$}}}

\def\scrN{{\mathcal N}}

\def\scrP{{\mathcal P}}

\def\scrS{{\mathcal S}}

\def\Im{\operatorname{Im}}

\def\e{\mathrm{e}}
\def\i{\mathrm{i}}

\def\C{\operatorname{C{}}}

\def\L{\operatorname{L{}}}

\def\SL{\operatorname{SL}}

\def\meas{\operatorname{meas}}

\def\GamG{\Gamma\backslash G}
\def\GamGG{\widehat\Gamma\backslash\widehat G}

\def\SLZ{\SL(2,\ZZ)}
\def\SLR{\SL(2,\RR)}

\def\trans{\,^\mathrm{t}\!}

\title{Limit theorems for skew translations}
\author{Jory Griffin}
\author{Jens Marklof}
\address{Jory Griffin, School of Mathematics, University of Bristol,
Bristol BS8 1TW, U.K.\newline
\rule[0ex]{0ex}{0ex} \hspace{8pt}{\tt jg9183@bristol.ac.uk}}
\address{Jens Marklof, School of Mathematics, University of Bristol,
Bristol BS8 1TW, U.K.\newline
\rule[0ex]{0ex}{0ex} \hspace{8pt}{\tt j.marklof@bristol.ac.uk}}
\date{17 September 2013}
\thanks{The research leading to these results has received funding from the European Research Council under the European Union's Seventh Framework Programme (FP/2007-2013) / ERC Grant Agreement n. 291147. J.M. is furthermore supported by a Royal Society Wolfson Research Merit Award}
\subjclass[2010]{37A50,11F27,60F05,37D40,28D05}

\begin{document}

\begin{abstract}
Bufetov, Bufetov-Forni and Bufetov-Solomyak have recently proved limit theorems for translation flows, horocycle flows and tiling flows, respectively. We present here analogous results for skew translations of a torus.

\end{abstract}

\maketitle

\newcounter{sectionno}\setcounter{sectionno}{0}

\refstepcounter{sectionno}\S \arabic{sectionno}. 
\label{intro}
Bufetov \cite{Bufetov13} has recently established limit theorems for translation flows on flat surfaces and, in joint work with Forni \cite{Bufetov12}, for horocycle flows on compact hyperbolic surfaces. Bufetov and Solomyak \cite{Bufetov13b} have proved analogous theorems for tiling flows. The striking feature of these results is that the central limit theorem (a common feature of many ``chaotic'' dynamical systems) fails. The limit laws are instead characterised in terms of the corresponding renormalisation dynamics. The purpose of the present note is to point out that an analogous result holds in the simpler case of skew translations of the torus. Our approach uses the modular invariance of theta sums as in \cite{Marklof99,MarklofIMRN,Marklof03}. The limit theorems proved in these papers may be interpreted as limit theorems for random, rather than fixed, skew translations, cf.~also \cite{Jurkat81,Jurkat82,Jurkat83} and \cite{Cellarosi11}. The approach by Flaminio and Forni \cite{Flaminio} developed for nilflows may yield an alternative route to our main result, but we have not explored this further. Another class of systems which exhibit non-normal limit laws are random translations of the torus; we refer the reader to Kesten's Cauchy limit theorem for circle rotations \cite{Kesten1,Kesten2} and the recent higher-dimensional generalisations by Dolgopyat and Fayad \cite{DF1,DF2}; see also Sinai and Ulcigrai \cite{Sinai} for limit theorems for circle rotations with non-integrable test functions.

\vspace{10pt}\refstepcounter{sectionno}\S \arabic{sectionno}.
Given $\alpha\in\RR$, a skew translation of the torus $\TT^2=\RR^2/\ZZ^2$ is defined by
\begin{equation}
\Lambda_\alpha:\TT^2 \rightarrow \TT^2 ,\qquad 
(p,q) \mapsto (p+\alpha, q+p).
\end{equation}
The $n$th iterate of this map is 
\begin{equation}
\Lambda_\alpha^n(p,q)=\bigg(p+n\alpha,q+np+ \frac{n(n-1)}{2}\,\alpha\bigg) .
\end{equation}
For $\alpha\notin\QQ$, the map is uniquely ergodic with respect to Lebesgue measure on $\TT^2$, and so for every $\psi\in\C(\TT^2,\CC)$ (the space of continuous functions $\TT^2\to\CC$) and $(p,q)\in\TT^2$
\begin{equation}\label{UE}
\lim_{N\to\infty} \frac{S_{N,\alpha} [\psi](p,q)}{N} = \int_{\TT^2} \psi(x) \, dx,
\end{equation}
with the Birkhoff sum
\begin{equation}
S_{N,\alpha} [\psi] = \sum_{n=1}^N \psi\circ \Lambda_\alpha^n .
\end{equation}
The rate of convergence in \eqref{UE} depends on how well $\alpha$ is approximable by rationals, see \cite{Flaminio} and \cite{Klopp} for the most recent results in this direction. By replacing $\psi$ with $\psi-\int \psi$, we may assume without loss of generality that $\psi$ has mean zero, so that the right hand side of \eqref{UE} vanishes. To keep the presentation simple, we assume throughout that $\psi$ is in fact a simple harmonic,
\begin{equation}
\psi(p,q)=\psi_{k,\ell}(p,q)=e(k p+\ell q), \qquad e(x):=\e^{2\pi\i x},
\end{equation}
for $k,\ell\in\ZZ$. The case $\ell=0$ leads to a geometric sum and is elementary. We therefore assume in the following that $\ell\neq 0$. 

We are interested in the distribution of the random variable
\begin{equation} \label{XNa}
X_{N,\alpha} = \frac{S_{N,\alpha} [\psi](p,q)}{\sqrt N}
\end{equation}
for large $N$, where $(p,q)$ is uniformly distributed in $\TT^2$ with respect to Lebesgue measure.  
The normalisation of the Birkhoff sum by $\sqrt N$ ensures that the variance of $X_{N,\alpha}$ is unity,
\begin{equation}
\EE |X_{N,\alpha}|^2 = 1.
\end{equation}
We also note that the distribution of $X_{N,\alpha}$ in $\CC$ is rotation-invariant.

Generalizing our results to arbitrary trigonometric polynomials $\psi$ is possible following the discussion in \cite{MarklofIMRN}, but will involve replacing $\SLR$ by its metaplectic cover, $\RR^2$ by the Heisenberg group and the modular group $\SLZ$ by certain congruence subgroups. The extension to more general functions can presumably be achieved by an analogue of the approximation argument used in \cite{MarklofIMRN}.

\vspace{10pt}\refstepcounter{sectionno}\S \arabic{sectionno}.
To state our main result, let $G=\SLR$, $\Gamma=\SLZ$ and
\begin{equation}
n(u)=\begin{pmatrix} 1 & u \\ 0 & 1\end{pmatrix},\qquad
\Phi^t = \begin{pmatrix} \e^{-t/2} & 0 \\ 0 & \e^{t/2} \end{pmatrix} .
\end{equation}
We recall that $\GamG$ can be identified with the unit tangent bundle of the modular surface and the action of $\{\Phi^t\}_{t\in\RR}$ and $\{n(u)\}_{u\in\RR}$ by right multiplication on $\GamG$ yields the geodesic and horocycle flow on the unit tangent bundle of the modular surface.

\begin{thm}\label{thm1}
Fix $\psi=\psi_{k,\ell}$ as above. Then, for any bounded $F\in\C(\CC,\RR)$ there exists $\nu[F]\in\C(\GamG,\RR)$ so that for any $\alpha\in\RR$, $N=\e^{t/2}\in\NN$, $k\in\ZZ$ and $\ell\in\ZZ\setminus\{0\}$,
\begin{equation}
\EE F(X_{N,\alpha}) = \nu[F](n(\ell\alpha)\Phi^t) .
\end{equation}
\end{thm}

For every $M\in\GamG$, the linear functional $\nu_M: F\mapsto \nu[F](M)$ defines a probability measure on $\CC$. We will see that 
\begin{equation}\label{defP}
\scrP=\{ \nu_M : M\in\GamG \} \cup \{\delta_0 \} 
\end{equation}
is the closure of the space of these probability measures (with respect to the weak topology), where $\delta_0$ is the delta mass at zero. The sequence of probability measures $\rho_{N,\alpha}$ describing the distribution of the $X_{N,\alpha}$ is therefore relatively compact, i.e., every sequence of $\rho_{N,\alpha}$ contains a converging subsequence. The following theorem, which is a corollary of Theorem \ref{thm1} and standard distribution properties of the geodesic flow on $\GamG$, provides a complete classification of all weak limits of subsequences in $(\rho_{N,\alpha})_N$ with $\alpha\in\RR$. We denote weak convergence by $\xrightarrow{w}$.

\begin{thm} \label{thm2}
\begin{enumerate}
\item[(i)] Fix $\alpha\in\RR$. Then, for every convergent subsequence $(\rho_{N_i,\alpha})_i$ there exists $\nu\in\scrP$ such that $\rho_{N_i,\alpha}\xrightarrow{w} \nu$.
\item[(ii)] There is a null set $\scrN\subset\RR$ such that for every $\alpha\in\RR\setminus\scrN$ the following holds. For every $\nu\in\scrP$ there exists a convergent subsequence $(\rho_{N_i,\alpha})_i$ such that $\rho_{N_i,\alpha}\xrightarrow{w} \nu$.
\end{enumerate}
\end{thm}

The fact that $\nu=\delta_0$ may arise as a limit of certain subsequences raises the question whether a better normalisation will lead to a less singular limit. We will return to this point in the final sections, \S \ref{final0} onward. 

\vspace{10pt}\refstepcounter{sectionno}\S \arabic{sectionno}.
For $f\in\scrS(\RR)$ (Schwartz class), $\tau=u+\i v\in\HH=\{z\in\CC:\Im z>0\}$, $\phi\in\RR$ and $\vecxi=\begin{pmatrix} x \\ y \end{pmatrix}\in\RR^2$, we define the theta function
\begin{equation}\label{Theta}
\Theta_f(\tau,\phi;\vecxi) =
v^{1/4} 
\sum_{n\in\ZZ} f_\phi( (n-y) v^{1/2}) 
e(\tfrac12(n-y)^2 u + n x) 
\end{equation}
with 
\begin{equation}
f_\phi(w) =
\begin{cases}
f(w) & (\phi=0\bmod 2\pi)\\[2mm]
f(-w) & (\phi=\pi\bmod 2\pi)\\[2mm]
\displaystyle
|\sin\phi|^{-1/2} 
\int_{\RR} e\left[\frac{\tfrac12(w^2+w'^2)\cos\phi-
w w'}{\sin\phi}\right]& f(w')\,dw'  \\
& (\phi\neq 0\bmod\pi).
\end{cases}
\end{equation}
The function $f_\phi$ is rapidly decaying for any $\phi$ if $f\in\scrS(\RR)$, and hence the series in \eqref{Theta} converges absolutely. The theta sum is connected to the Birkhoff sum via
\begin{equation}
\frac{S_{N,\alpha}[\psi_{k,\ell}](p,q)}{\sqrt N} = \Theta_\chi\bigg(\ell \alpha+\i N^{-2},0 ; \begin{pmatrix} (k-\frac\ell2) \alpha + \ell p \\ 0 \end{pmatrix}\bigg) \; \psi_{k,\ell}(p,q),
\end{equation}
where $\chi$ is the characteristic function of $(0,1]$, which is obviously not in the Schwartz class. Making sense of $\Theta_\chi$ for general values of $\phi$ will be one of the objectives of our analysis. 

\vspace{10pt}\refstepcounter{sectionno}\S \arabic{sectionno}.
The point of using more variables ($\phi$ and $y$) than we originally need is that the function $\Theta_f(\tau,\phi;\vecxi)$ is a natural automorphic function for $f\in\scrS(\RR)$. The following lemma shows that we only need to consider the absolute value of $\Theta_f$, which will lead to a simpler description of the invariance properties of the theta sum, cf.~\cite[Section 4.5]{Marklof03}.

\begin{lem}\label{lem1}
Let $f=\chi$ be the characteristic function of $(0,1]$, $\alpha\in\RR$ and $N\in\NN$. For $(p,q)$ uniformly distributed in $\TT^2$, the random variable 
\begin{equation}\label{rv1}
\frac{S_{N,\alpha}[\psi_{k,\ell}](p,q)}{\sqrt N}\end{equation}
has the same distribution as the random variable
\begin{equation}\label{rv2}
e(t) \bigg| \Theta_\chi\bigg(\ell\alpha+\i N^{-2},0 ; \begin{pmatrix} x \\ y \end{pmatrix} \bigg)  \bigg|,
\end{equation}
where $(t,x,y)$ is uniformly distributed in $\TT^3$.
\end{lem}

\begin{proof}
The invariance of Lebesgue measure under translations shows that the random variable \eqref{rv1} has the same distribution as (set $u=\ell\alpha$)
\begin{equation}
N^{-1/2} e(t)   \bigg|\sum_{0<n\leq N}  e(\tfrac12 n^2 u + n x)  \bigg|
\end{equation}
which in turn has the same distribution as 
\begin{equation}
N^{-1/2} e(t)   \bigg|\sum_{0<n\leq N}  e(\tfrac12(n-y)^2 u + n x)  \bigg|
\end{equation}
for every fixed $y$. The lemma follows from the observation that 
\begin{equation}
\sum_{0<n\leq N}  e(\tfrac12(n-y)^2 u + n x) =\sum_{0<n-y\leq N}  e(\tfrac12(n-y)^2 u + n x) 
\end{equation}
for $0\leq y< 1$.
\end{proof}

\vspace{10pt}\refstepcounter{sectionno}\S \arabic{sectionno}.
To state the invariance properties of $|\Theta_f|$, define the semi-direct product
group
$$
\widehat G=\SLR \ltimes \RR^{2}
$$
with multiplication law
$$
(M;\vecxi)(M';\vecxi') =(MM';\vecxi+M\vecxi') .
$$
We identify $\SLR$ with $\HH\times [0,2\pi)$ as usual via the Iwasawa decomposition
\begin{equation}\label{iwasawa}
M = \begin{pmatrix} 1 & u \\ 0 & 1 \end{pmatrix}
\begin{pmatrix} v^{1/2} & 0 \\ 0 & v^{-1/2} \end{pmatrix}
\begin{pmatrix} \cos\phi & -\sin\phi \\ \sin\phi & \cos\phi \end{pmatrix} \mapsto (u+\i v,\phi) .
\end{equation}
With this identification, group multiplication in $\SLR$ becomes the group action of $\SL(2,\RR)$ on $\HH\times[0,2\pi)$,
\begin{equation}\label{iwaact}
\begin{pmatrix}
a & b \\
c & d
\end{pmatrix}
(\tau,\phi)= \bigg(\frac{a\tau+b}{c\tau+d}, \phi+\arg(c\tau+d) \bmod2\pi\bigg) .
\end{equation}
For $\gamma\in\SLR$, we use the notation $(u_\gamma+\i v_\gamma,\phi_\gamma):=\gamma(\tau,\phi)$
and $\gamma\tau:=u_\gamma+\i v_\gamma$.

The following result is well known; for an explicit proof see  \cite{Marklof03}.

\begin{thm}
For $f\in\scrS(\RR)$, the function $|\Theta_f(M;\vecxi)|=|\Theta_f(\tau,\phi;\vecxi)|$ is continuous on $\widehat G$ and invariant under the left action of the discrete subgroup
\begin{equation}
\widehat\Gamma=\left\{ 
\bigg( \begin{pmatrix}
a & b \\
c & d
\end{pmatrix}
; 
\begin{pmatrix}
\frac12 ab \\
\frac12 cd 
\end{pmatrix}
+\vecm \bigg)
: \; \begin{pmatrix}
a & b \\
c & d
\end{pmatrix} \in\SL(2,\ZZ),\; \vecm\in\ZZ^{2}
\right\} \subset \widehat G.
\end{equation}
\end{thm}

In other words, we may view $|\Theta_f|$ as a continuous function on $\GamGG$. This implies that, for any bounded $F\in\C(\CC,\RR)$, the function 
\begin{equation}\label{nuf}
\nu_f[F]: \GamG\to \RR,\qquad M\mapsto\int_{\TT^3} F\big(e(t)|\Theta_f(M;\vecxi)|\big) \, dt\,d\vecxi
\end{equation}
is bounded continuous. 
The following theorem notes that the same statement is true when we choose a more singular $f$.

\begin{thm}\label{thm:main}
For $f\in\L^2(\RR)$ and bounded $F\in\C(\CC,\RR)$, the function $\nu_f[F]$ is bounded continuous.
\end{thm}

\vspace{10pt}\refstepcounter{sectionno}\S \arabic{sectionno}.
We will split the proof of Theorem \ref{thm:main} into several steps. 

\begin{lem}
For $f\in\L^2(\RR)$
\begin{equation}\label{cv1}
\int_{\TT^2} |\Theta_f(M;\vecxi)|^2 \, d\vecxi =\| f \|_2^2.
\end{equation}
and
\begin{equation}\label{cv2}
\int_{\TT^2} |\Theta_f(M;\vecxi)| \, d\vecxi \leq \| f \|_2.
\end{equation}
\end{lem}

\begin{proof}
The left hand side of \eqref{cv1} equals by Parseval's identity
\begin{equation}
v^{1/2}
\sum_{n\in\ZZ} \int_\TT  |f_\phi( (n-y) v^{1/2}) |^2 dy =\|f_\phi\|_2^2,
\end{equation}
and the claim follows from the fact that $\|f_\phi\|_2=\| f\|_2$ (see e.g.~\cite{Marklof03}). Inequality \eqref{cv2} follows from H\"older's inequality.
\end{proof}

\begin{lem}
Let $F\in\C_0^\infty(\CC,\RR)$ (infinitely differentiable and compactly supported). Then there exists a constant $C>0$ such that for all $f, \tilde f\in\L^2(\RR)$ and all $M\in\GamG$
\begin{equation}
|\nu_f[F](M)-\nu_{\tilde f}[F](M)|< C \| f-\tilde f \|_2 . 
\end{equation}
\end{lem}

\begin{proof}
For every $F \in \C_0^\infty(\CC,\RR)$ we can find a $C >0$ such that
\begin{equation}
| F(z_1) - F(z_2) | \leq C |z_1-z_2|
\end{equation}
for all $z_1,z_2\in\CC$.
Hence
\begin{equation}
\begin{split}
|\nu_f[F](M)-\nu_{\tilde f}[F](M)| & \leq \int_{\TT^3} \big| F\big(e(t)\Theta_f(M;\vecxi)\big) - F\big(e(t)\Theta_{\tilde f}(M;\vecxi)\big) \big|\, dt\,d\vecxi \\
& \leq C \int_{\TT^2} \big| \Theta_f(M;\vecxi)- \Theta_{\tilde f}(M;\vecxi) \big| \,d\vecxi \\
& = C \int_{\TT^2} \big| \Theta_{f-\tilde f}(M;\vecxi) \big| \,d\vecxi \\
& \leq C \| f-\tilde f \|_2
\end{split}
\end{equation}
by \eqref{cv2}.
\end{proof}

\begin{lem}\label{lem4}
For all $f\in\L^2(\RR)$, bounded $F \in \C(\CC,\RR)$ and $M\in\GamG$,
\begin{equation}
|\nu_f[F](M)| \leq \sup_{z\in\CC} |F(z)| .
\end{equation}
\end{lem}

\begin{proof}
This is immediate from \eqref{nuf}.
\end{proof}

\begin{lem}
Let $R>0$ and $F\in\C(\CC,\RR)$ be bounded, such that $F(z)=0$ for all $|z|\leq R$. Then, for all $f\in\L^2(\RR)$,
\begin{equation}
 |\nu_f[F](M)| < R^{-2} \|f\|_2^2 \sup_{z\in\CC}|F(z)| .
\end{equation}
\end{lem}

\begin{proof}
We have
\begin{equation}
\begin{split}
|\nu_f[F](M)|  & \leq \sup_{z\in\CC} |F(z)| \meas\{ \vecxi\in\TT^2 : |\Theta_f(M;\vecxi)| > R \} \\
& < R^{-2} \sup_{z\in\CC} |F(z)| \int_{\TT^2} |\Theta_f(M;\vecxi)|^2 \, d\vecxi ,
\end{split}
\end{equation}
which proves the claim in view of \eqref{cv1}.
\end{proof}

\begin{proof}[Proof of Theorem \ref{thm:main}]
Let $F$ be bounded continuous and $f\in\L^2(\RR)$. In view of Lemma \ref{lem4}, it remains to be shown that $\nu_f[F](M)$ is continuous. Given $R>0$, we write $F=F_0+F_1+F_2$ so that
\begin{itemize}
\item $F_0\in\C_0^\infty(\CC,\RR)$ with support in the disc of radius $2R$;
\item $F_1\in\C(\CC,\RR)$ with $\sup |F_1|\leq R^{-1}$;
\item $F_2\in\C(\CC,\RR)$ bounded and supported outside the disk of radius $R$. 
\end{itemize}
We may furthermore choose $\tilde f\in\scrS(\RR)$ such that $\| f-\tilde f\|_2 < R^{-1}$.
The above lemmas imply that (by taking $R$ sufficiently large) that, given any $\epsilon>0$, there exists $\tilde f\in\scrS(\RR)$ such that
\begin{equation}\label{key:in}
\sup_{M\in\GamG} |\nu_f[F](M)-\nu_{\tilde f}[F](M)|< \epsilon .
\end{equation}
The continuity of $\nu_f[F](M)$ thus follows from the continuity of $\nu_{\tilde f}[F](M)$.
\end{proof}

\vspace{10pt}\refstepcounter{sectionno}\S \arabic{sectionno}.
With the choice $\nu[F]=\nu_\chi[F]$, Lemma \ref{lem1} and Theorem \ref{thm:main} imply Theorem \ref{thm1} (recall that in Iwasawa coordinates \eqref{iwasawa} we have $n(\ell\alpha)\Phi^t=(\ell\alpha +\i N^{-2},0)$ with $N=\e^{t/2}$). 

We will now conclude our study by explaining the remarks leading up to Theorem \ref{thm2} and its proof.

\begin{lem}\label{last:lem}
\begin{equation}
\overline{\{ \nu_M : M\in\GamG \}} =\scrP .
\end{equation}
\end{lem}

\begin{proof}
By the continuity of $\nu[F]=\nu_\chi[F]$, all that remains is to show that if $(M_i)_i$ is a divergent sequence in $\GamG$ (i.e., given any compact $K\subset\GamG$ there is $i_0$ such that $M_i\notin K$ for all $i\geq i_0$) then $\nu_{M_i}\xrightarrow{w}\delta_0$. In terms of the Iwasawa coordinates \eqref{iwasawa}, we can parametrize such $M_i$ by $M_i=(\tau_i,\phi_i)$, where divergence is equivalent to $v_i=\Im\tau_i\to\infty$. 
In this limit we have (cf.~\cite[Prop.~4.10]{Marklof03}) for $f\in\scrS(\RR)$
\begin{equation}
|\Theta_f(M_i,\vecxi)| = v_i^{1/4} \sum_{n\in\ZZ} |f_{\phi_i}((n-y) v_i^{1/2}) | + O_T(v_i^{-T})
\end{equation}
for any $T\geq 1$ and uniformly in $\vecxi=\trans(x,y)\in\TT^2$. Hence  
\begin{equation}
\nu_f[F](M_i)= \int_{y\notin[-v_i^{-1/4},v_i^{-1/4}]+\ZZ} F\big(e(t)|\Theta_f(M;\vecxi)|\big) \, dt\,d\vecxi +O(v_i^{-1/4}) \to F(0) ,
\end{equation}
by the rapid decay of $f$. Inequality \eqref{key:in} shows that this also holds for $f=\chi$, i.e., $\nu[F](M_i)\to F(0)$.
\end{proof}

Theorem \ref{thm:main} and Lemma \ref{last:lem} complete the proof of Theorem \ref{thm2} (i). Theorem \ref{thm2} (ii) follows from the well known fact that, for almost every $u\in\RR$, the orbit $\{ \Gamma n(u) \Phi^t \}_{t\geq 0}$ is dense in $\GamG$.

\vspace{10pt}\refstepcounter{sectionno}\S \arabic{sectionno}.\label{final0}
We now return to the question in \S \ref{intro}, as to whether the appearance of $\delta_0$ as a weak limit in Theorem \ref{thm2} can be avoided by a different normalisation. We can give an immediate answer when $\alpha=0$. In this case
\begin{equation}
\Lambda_0(p,q)=(p,q+p) 
\end{equation}
becomes a rotation in the second component; in fact a random rotation since $(p,q)$ are uniformly distributed in $\TT^2$. The Birkhoff sum is thus
\begin{equation}
S_{N,0} [\psi_{k,\ell}](p,q) = \sum_{n=1}^N e(k p +\ell (q+np)) = e((k+\ell) p+\ell q)\, \frac{1-e(N\ell p)}{1-e(\ell p)} .
\end{equation}
Considered as a random variable, $S_{N,0} [\psi_{k,\ell}]$ has the same distribution as
\begin{equation}
e(t)\, \frac{1-e(N x)}{1-e(x)} ,
\end{equation}
with $(t,x)$ uniformly distributed in $\TT^2$, 
and therefore convergences in distribution, {\em without any normalisation} as $N\to\infty$, to the random variable
\begin{equation}\label{firstchoice}
e(t)\, \frac{1-e(y)}{1-e(x)} ,
\end{equation}
with $(t,x,y)$ uniformly distributed in $\TT^3$. Note that the density of the radial distribution of \eqref{firstchoice} is
\begin{equation}
\frac{2}{\pi^2 r} \log \bigg|\frac{1+r}{1-r}\bigg| .
\end{equation}

Our Theorem \ref{thm2}, on the other hand, produces only the weaker statement that $N^{-1/2} S_{N,0} [\psi_{k,\ell}]$ converges in distribution to $0$. What is more, if we chose the test function $\psi(p,q)=\chi_A(q)-|A|$, where $\chi_A$ is the characteristic function of an interval $A\subset \TT$ of length $|A|$, Kesten's theorem \cite{Kesten1,Kesten2} states that 
$(\log N)^{-1} S_{N,0} [\psi]$
convergences in distribution to a Cauchy law. 

\vspace{10pt}\refstepcounter{sectionno}\S \arabic{sectionno}.\label{final}
Let us now turn to the question of the normalisation of $S_{N,\alpha} [\psi]$ for general $\alpha$, and $\psi=\psi_{k,\ell}$. 

For $u\in\RR$, define
\begin{equation}
c_N(u) = \min\{ c\in \NN : \|cu\|\leq N^{-1}\} ,
\end{equation}
where $\|\,\cdot\,\|$ denotes the distance to the nearest integer.
Dirichlet's box principle implies that $c_N(u)\leq N$. We also have $\liminf_{N\to\infty} \frac{c_N(u)}{N}=0$ unless $u$ is of bounded type.

Instead of $X_{N,\alpha}$ as in \eqref{XNa}, we now consider the random variable
\begin{equation}
\widetilde X_{N,\alpha} = \frac{S_{N,\alpha} [\psi_{k,\ell}](p,q)}{\sqrt{c_N(\ell\alpha)}} ;
\end{equation}
denote by $\widetilde\rho_{N,\alpha}$ its probability measure. 
We will show that the limit distributions that arise as certain subsequences of $\widetilde X_{N,\alpha}$ are described by the family of random variables parametrized by $\omega,\varphi\in\RR$,
\begin{equation}\label{Ydef}
Y_{\omega,\varphi}= \frac{e(t)}{2\pi}  \bigg| \sum_{n\in\ZZ} \frac{1+e(t'-n\varphi)}{n-y} \; e(\tfrac12 (n-y)^2 \omega + nx) \bigg| ,
\end{equation}
where $(t,t',x,y)$ are uniformly distributed in $\TT^4$. The probability measure associated with $Y_{\omega,\varphi}$ is denoted by $\nu_{\omega,\varphi}$. Note that $Y_{\omega,\varphi}$ has the same distribution as $Y_{\omega+1,\varphi}$ and $Y_{\omega,\varphi+1}$. That is,
\begin{equation}
\nu_{\omega,\varphi}=\nu_{\omega+1,\varphi}=\nu_{\omega,\varphi+1}.
\end{equation}

\vspace{10pt}\refstepcounter{sectionno}\S \arabic{sectionno}.
Let $d_N(u)$ be the unique integer so that $-\frac12\leq c_N(u)u + d_N(u)<\frac12 $. It follows from the definition of $c_N(u)$ that $\gcd(c_N(u),d_N(u))=1$. We denote by $a_N(u)$ the inverse of $d_N(u)$ mod $c_N(u)$. We write $\{ x\}$ for the fractional part of $x\in\RR$.

\begin{thm}\label{thm500}
Fix $u=\ell\alpha\in\RR$ and consider a subsequence $( N_i )_{i\in\NN}$ such that 
\begin{equation}\label{bau}
\bigg\{ \frac{N_i}{c_{N_i}(u)} \bigg\} \to\varphi ,\qquad\bigg\{ \frac{a_{N_i}(u)}{c_{N_i}(u)}\bigg\}\to\omega ,
\end{equation}
and 
\begin{equation}\label{baubau}
\frac{c_{N_i}(u)}{N_i}\to 0,  \qquad \limsup_{i\to\infty} \frac{N_i^2 \| c_{N_i}(u) u\|}{c_{N_i}(u)} <\infty.
\end{equation}
Then $\widetilde\rho_{N_i,\alpha}\xrightarrow{w} \nu_{\omega,\varphi}$.
\end{thm}

Note that for the above sequence we have $\rho_{N_i,\alpha}\xrightarrow{w} \delta_0$ (in the original $\sqrt N$ normalisation of Theorem \ref{thm2}). 

For every $u$ which is not of bounded type, there is a sequence $N_i$ for which assumption \eqref{baubau} is satisfied. A geometric justification of this claim follows from Lemma \ref{lem77} below. Hypothesis \eqref{bau} can be obtained by passing to a suitable subsequence.

Figures \ref{fig1} and \ref{fig2} show the value distribution of $|\widetilde X_{N,\alpha}|$ for $\alpha=u=\pi-3$ and two different values of $N$, compared with the distribution of $|Y_{\omega,\varphi}|$.

\begin{figure}
\begin{center}
\includegraphics[width=0.7\textwidth]{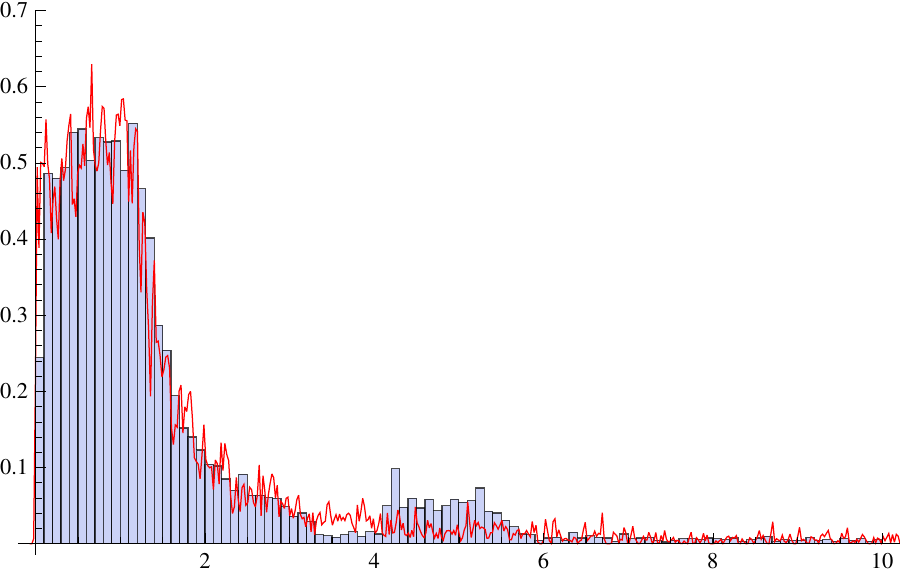}
\end{center}
\caption{The histogram displays the value distribution of $|\widetilde X_{N,\alpha}|$ for $\alpha=u=\pi-3$ and $N=2260$, where $p$ is sampled over $10,000$ points in $[0,1]$. We have $c_N(u)=113$, $d_N(u)=-16$ and $a_N=7$, and furthermore $\frac{N^2 \| c_{N}(u) u\|}{c_{N}(u)}\approx 1.363$, $\frac{c_N(u)}{N}=0.05$.
The solid curve is the distribution of the random variable $|Y_{\omega,\varphi}|$ for $\omega=7/113$ and $\varphi=\{ 2260/113\}=0$ sampled at random points $t',x,y\in[0,1]$, with $10,000$ values each. The series \eqref{Ydef} defining $Y_{\omega,\varphi}$ is truncated at $n=\pm 1000$. The first distribution is plotted using Mathematica's {\tt Histogram} command with bin width 0.1, the second using the {\tt SmoothHistogram} command with a Gaussian kernel and bandwidth 0.01.} 
\label{fig1}
\end{figure}
\begin{figure}
\begin{center}
\includegraphics[width=0.7\textwidth]{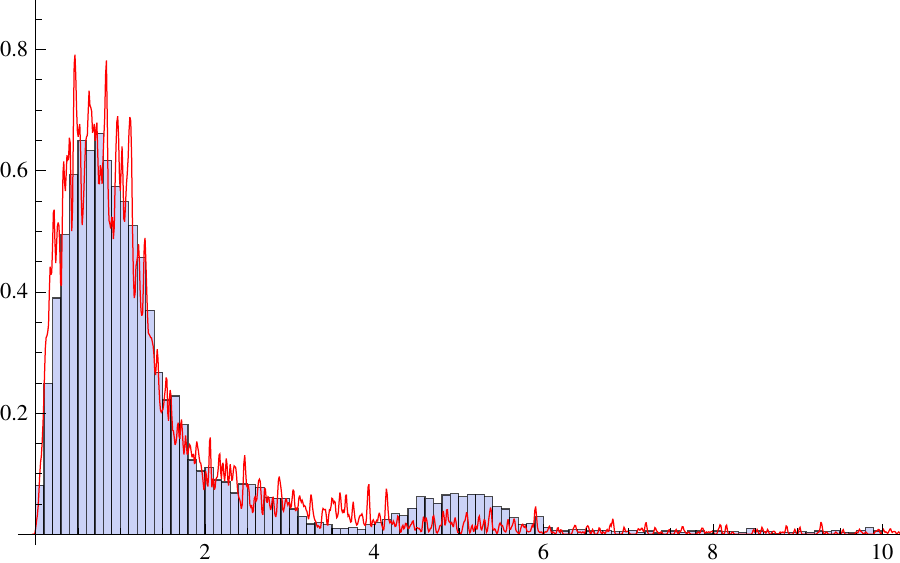}
\end{center}
\caption{The same histogram as in Figure \ref{fig1}, now for $N=2300$. We have again $c_N(u)=113$, $d_N(u)=-16$ and $a_N=7$, and furthermore $\frac{N^2 \| c_{N}(u) u\|}{c_{N}(u)}\approx 1.411$, $\frac{c_N(u)}{N}\approx 0.0491$.
The solid curve is the distribution of the random variable $|Y_{\omega,\varphi}|$ for $\omega=7/113$ and $\varphi=\{2300/113\}\approx 0.354$.} \label{fig2}
\end{figure}

\vspace{10pt}\refstepcounter{sectionno}\S \arabic{sectionno}.
The proof of Theorem \ref{thm500} will require the following lemmas.

\begin{lem}\label{lem77}
Let $(N_i)_{i\in\NN}$ as in Theorem \ref{thm500}, and set
$c_i=c_{N_i}(u)$, $d_i=d_{N_i}(u)$ and $a_i=a_{N_i}(u)$. Then there exist $b_i\in\ZZ$ such that
\begin{equation}
\gamma_i=\begin{pmatrix} a_i & b_i \\ c_i & d_i \end{pmatrix} \in\SLZ .
\end{equation}
For $(u_i+\i v_i,\phi_i):=\gamma_i(u+\i N_i^{-2},0)$ we have the relations
\begin{equation}\label{S1}
v_i^{-1/2}\cos\phi_i = N_i(c_iu+d_i),\qquad v_i^{-1/2}\sin\phi_i = \frac{c_i}{N_i},
\end{equation}
\begin{equation}\label{S2}
u_i + v_i\cot\phi_i =\frac{a_i}{c_i}
\end{equation}
and 
\begin{equation} \label{S3}
v_i\to\infty, \qquad \liminf_{i\to\infty} \sin\phi_i >0 .
\end{equation}
\end{lem}

\begin{proof}
Since $a_i$ is the inverse of $d_i\bmod c_i$, $a_id_i-1$ is divisible by $c_i$ and thus $\det\gamma_i=a_id_i-b_ic_i=1$ has an integer solution $b_i$. Relations \eqref{S1} and \eqref{S2} follow from \eqref{iwaact} by direct computation. As to \eqref{S3}, we have
\begin{equation} \label{sini}
\sin\phi_i = \frac{c_i}{N_i} v_i^{1/2} = \bigg(1+ \frac{N_i^4}{c_i^2}  (c_iu+d_i)^2\bigg)^{-1/2} ,
\end{equation}
which is bounded away from zero by assumption \eqref{baubau}.
\end{proof}

The above proof shows that \eqref{S3} is in fact equivalent to \eqref{baubau}. It is a well known fact that if $u$ is not of bounded type, then the geodesic $\Gamma \{ (u+\i\e^{-t},0) : \RR_{\geq 0}\}$ is unbounded in $\GamG$ and therefore there is a subsequence of times $t_i\geq 0$ so that $\Gamma (u+\i\e^{-t_i},0)=\Gamma(u_i+\i v_i,\phi_i)$ with $v_\i\to\infty$ and $\sin\phi_i$ bounded away from zero. This shows that there exist subsequences $N_i$ satisfying \eqref{sini} and thus \eqref{baubau} for all $u\in\RR$ that are not of bounded type.

\begin{lem}\label{lem:chi}
Given $\phi_0>0$, there is a constant $C>0$ such that for all $\phi\notin [-\phi_0,\phi_0]+\pi\ZZ$, $w\in\RR_{>0}$,
\begin{equation}
\big|\chi_\phi(w) - \chi_\phi^{(0)}(w)\big| \leq C w^{-2},
\end{equation}
where
\begin{equation}
\chi_\phi^{(0)}(w)= \epsilon_\phi |\sin\phi|^{1/2} e(\tfrac12 w^2 \cot\phi) \frac{1-e(\tfrac12 \cot\phi-\frac{w}{\sin\phi})}{2\pi\i w}
\end{equation}
and $\epsilon_\phi=1$ if $\phi\in[0,\pi)+2\pi\ZZ$ and $\epsilon_\phi=-1$ if $\phi\in[\pi,2\pi)+2\pi\ZZ$.
\end{lem}

\begin{proof}
We have
\begin{equation}
\begin{split}
\chi_\phi(w) = & |\sin\phi|^{-1/2} 
\int_0^1 e\left[\frac{\tfrac12(w^2+w'^2)\cos\phi-
w w'}{\sin\phi}\right]\,dw' \\
= & - \frac{\epsilon_\phi |\sin\phi|^{1/2} }{2\pi\i w}  \bigg\{ 
e\left[\frac{\tfrac12(w^2+w'^2)\cos\phi-
w w'}{\sin\phi}\right] \bigg|_{w'=0}^1 \\
&  -
2\pi\i\cot\phi \int_0^1 e\left[\frac{\tfrac12(w^2+w'^2)\cos\phi-
w w'}{\sin\phi}\right]\,w'\,dw' \bigg\} .
\end{split}
\end{equation}
Using again integration by parts shows that the above integral is $O(w^{-1})$:
\begin{equation}
\begin{split}
\bigg| \int_0^1  & e\left[\frac{\tfrac12(w^2+w'^2)\cos\phi-
w w'}{\sin\phi}\right]\,w'\,dw' \bigg|\\
= & \frac{|\sin\phi|}{2\pi |w|} \bigg| e\left[\frac{\tfrac12(w^2+w'^2)\cos\phi-
w w'}{\sin\phi}\right]\,w'\bigg|_{w'=0}^1 \\
& -
\int_0^1 e\left[\frac{\tfrac12(w^2+w'^2)\cos\phi-
w w'}{\sin\phi}\right]\,(1+2\pi\i  w'^2\cot\phi)\,dw' \bigg| \\
\leq & \frac{1}{2\pi |w|} \bigg| 1+  \int_0^1(1+2\pi  w'^2)\,dw' \bigg| =O(w^{-1}) .
\end{split}
\end{equation}
\end{proof}

\begin{lem}\label{lem99}
Given $\phi_0,\eta>0$ there is a constant $\widetilde C>0$ such that for all $u\in\RR$, $v\in\RR_{>0}$, $\phi\notin [-\phi_0,\phi_0]+\pi\ZZ$, $y\notin [-\eta,\eta]+\ZZ$,
\begin{multline}\label{appro}
\int_\TT \bigg| \Theta_\chi\bigg(u+\i v ,\phi;\begin{pmatrix} x \\ y \end{pmatrix}\bigg) \\
- v^{1/4} \sum_{n\in\ZZ} \chi_\phi^{(0)}\big((n-y) v^{1/2}\big) \; e\big(\tfrac12(n-y)^2 u + n x\big) \bigg|^2 dx
\leq \widetilde C v^{-3/2} .
\end{multline}
\end{lem}

\begin{proof}
Parseval's identity shows that the left hand side of \eqref{appro} equals
\begin{equation}
v^{1/2} \sum_{n\in\ZZ} \big|  \chi_\phi((n-y) v^{1/2}) - \chi_\phi^{(0)}((n-y) v^{1/2}) \big|^2
\end{equation}
which, by Lemma \ref{lem:chi}, is less or equal to $C^2 v^{-3/2} \sum_{n\in\ZZ} (n-y)^{-4}$.
\end{proof}

\vspace{10pt}\refstepcounter{sectionno}\S \arabic{sectionno}.
We conclude by outlining the remaining steps in the proof of Theorem \ref{thm500}. Lemma \ref{lem1} shows that $\widetilde X_{N,\alpha}$ has the same distribution as the random variable
\begin{equation}\label{rv25}
e(t) \bigg(\frac{N}{c_N(\ell\alpha)}\bigg)^{1/2} \bigg| \Theta_\chi\bigg(\ell\alpha+\i N^{-2},0 ; \begin{pmatrix} x \\ y \end{pmatrix} \bigg)  \bigg|,
\end{equation}
where $(t,x,y)$ is uniformly distributed in $\TT^3$. In view of Theorem \ref{thm1}, the random variable \eqref{rv25} has, for $N=N_i$, the same distribution as 
\begin{equation}\label{rv26}
e(t) \frac{v_i^{1/4}}{|\sin\phi_i|^{1/2}} \bigg| \Theta_\chi\bigg(u_i+\i v_i,\phi_i ; \begin{pmatrix} x \\ y \end{pmatrix} \bigg)  \bigg|,
\end{equation}
with $u=\ell\alpha$ and $u_i,v_i,\phi_i$ as in Lemma \ref{lem77}. 
Lemma \ref{lem99} implies via Chebyshev's inequality that the distribution of \eqref{rv26} is arbitrarily close (as $v_i\to\infty$ with $\sin\phi_i$ bounded away from zero) to the distribution of 
\begin{multline}
e(t)  \frac{v_i^{1/2}}{|\sin\phi_i|^{1/2}}\bigg| \sum_{n\in\ZZ} \chi_{\phi_i}^{(0)}\big((n-y) v_i^{1/2}\big) \; e\big(\tfrac12(n-y)^2 u_i + n x\big) \bigg|  \\
=
e(t) \bigg| \sum_{n\in\ZZ} \frac{1-e(\tfrac12 \cot\phi_i-\frac{(n-y) v_i^{1/2}}{\sin\phi_i})}{2\pi\i (n-y)}  \; e\big(\tfrac12(n-y)^2 (u_i+v_i \cot\phi_i) + n x\big) \bigg| ,
\end{multline}
which in turn is arbitrarily close to the distribution of
\begin{equation}
e(t) \bigg| \sum_{n\in\ZZ} \frac{1-e(t'-\frac{n v_i^{1/2}}{\sin\phi_i})}{2\pi\i (n-y)}  \; e\big(\tfrac12(n-y)^2 (u_i+v_i \cot\phi_i) + n x\big) \bigg| ,
\end{equation}
with $t'$ uniformly distributed on $\TT$. This yields Theorem \ref{thm500}. \qed

In the particular case $u=0$ the limit distribution in Theorem \ref{thm500} is given by the random variable
\begin{equation}\label{Y00}
Y_{0,0}= e(t)\bigg| \frac{1+e(t')}{2\pi}  \sum_{n\in\ZZ} \frac{e(nx)}{n-y} \bigg| .
\end{equation}
It is a nice exercise (i) to show that the distribution of $Y_{0,0}$ remains the same if we condition $x$ to be any fixed non-integer value, and (ii) to prove directly that $Y_{0,0}$ indeed has the same distribution as the random variable \eqref{firstchoice} discussed in \S\ref{final0}. [Hint: Choose $x=\frac12$ in \eqref{Y00}.]


\begin{thebibliography}{99}

\bibitem{Bufetov13}
A. Bufetov, Limit theorems for translation flows. Ann. of Math. 179 (2014) 431--499.

\bibitem{Bufetov12}
A. Bufetov and G. Forni, Limit theorems for horocycle flows, arXiv:1104.4502

\bibitem{Bufetov13b}
A. Bufetov and B. Solomyak, Limit theorems for self-similar tilings, Comm. Math. Phys. 319 (2013) 761--789.

\bibitem{DF1}
D. Dolgopyat and B. Fayad, Deviations of ergodic sums for toral translations I. Convex bodies, Geom. Funct. Anal. 24 (2014) 85--115.

\bibitem{DF2}
D. Dolgopyat and B. Fayad,
Deviations of ergodic sums for toral translations II. Boxes, arXiv:1211.4323 

\bibitem{Cellarosi11}
F. Cellarosi, Limiting curlicue measures for theta sums, Ann. Inst. Henri Poincar\'e Probab. Stat. 47 (2011) 466--497.

\bibitem{Klopp}
A. Fedotov and F. Klopp,  An exact renormalization formula for Gaussian exponential sums and applications. Amer. J. Math. 134 (2012), no. 3, 711--748. 
 
\bibitem{Flaminio}
L. Flaminio and G. Forni, Equidistribution of nilflows and applications to theta sums. Ergodic Theory Dynam. Systems 26 (2006), no. 2, 409--433. 

\bibitem{Jurkat81}
W.B. Jurkat and J.W. van Horne, The proof of the central limit theorem for theta sums, Duke Math. J. 48 (1981) 873--885.

\bibitem{Jurkat82}
W.B. Jurkat and J.W. van Horne, On the central limit theorem for theta series, Michigan Math. J. 29 (1982) 65--67.

\bibitem{Jurkat83}
W.B. Jurkat and J.W. van Horne, The uniform central limit theorem for theta sums, Duke Math. J. 50 (1983) 649--666.

\bibitem{Kesten1}
H. Kesten, Uniform distribution mod 1. Ann. of Math.  71 (1960) 445--471. 

\bibitem{Kesten2}
H. Kesten, Uniform distribution mod 1. II. Acta Arith. 7 (1961/1962) 355--380. 
 
\bibitem{Marklof99}
J. Marklof, Limit theorems for theta sums, Duke Math. J. 97 (1999) 127--153 

\bibitem{MarklofIMRN}
J. Marklof, Almost modular functions and the distribution of $n^2 x$ modulo one, Int. Math. Res. Notices 39 (2003) 2131--2151
 
\bibitem{Marklof03}
J. Marklof, Pair correlation densities of inhomogeneous quadratic forms, Ann. of Math. 158 (2003) 419--471 

\bibitem{Sinai}
Ya.G. Sinai and C. Ulcigrai, A limit theorem for Birkhoff sums of non-integrable functions over rotations. Geometric and probabilistic structures in dynamics, 317--340, Contemp. Math., 469, Amer. Math. Soc., Providence, RI, 2008.
\end{thebibliography}
\end{document}